\documentclass[graybox]{svmult}

\def\Dbar{\leavevmode\lower.6ex\hbox to 0pt{\hskip-.23ex \accent"16\hss}D}

\def\bZ{{\mbox{\bf Z}}}
\def\bC{{\mbox{\bf C}}}

\def\paf{{\mbox{\rm PAF}}}
\def\psd{{\mbox{\rm PSD}}}
\def\dft{{\mbox{\rm DFT}}}

\makeindex             

\begin{document}

\title*{Periodic Golay pairs of length 72}

\author{Dragomir {\v{Z}}. {\Dbar}okovi{\'c} and Ilias S. Kotsireas}
\institute{Dragomir {\v{Z}}. {\Dbar}okovi{\'c} 
\at University of Waterloo, Department of Pure Mathematics, Waterloo, Ontario, N2L 3G1, Canada
\email{djokovic@math.uwaterloo.ca}
\and
Ilias S. Kotsireas 
\at  Wilfrid Laurier University, Department of Physics
\& Computer Science, Waterloo, Ontario, N2L 3C5, Canada 
\email{ikotsire@wlu.ca}}
%
%

\maketitle

\centerline{ {\bf Dedicated to Hadi Kharaghani on his 70th birthday} }
\vspace {1 cm}

\abstract{
We construct supplementary difference sets (SDS) with parameters  $(72;36,30;30)$. These SDSs give periodic Golay pairs of length 72. No periodic Golay pair of length 72 was known previously.
The smallest undecided order for periodic Golay pairs is now 90. The periodic Golay pairs constructed here are the first examples
having length divisible by a prime congruent to 3 modulo 4. 
The main tool employed is a recently introduced compression method. We observe that Turyn's multiplication of Golay pairs can  be also used to multiply a Golay pair and a periodic Golay pair. }

\section{Introduction}

Let $v$ be any positive integer. We say that a sequence $A=[a_0,a_1,\ldots,a_{v-1}]$ is a {\em binary sequence} if  $a_i\in\{1,-1\}$ for all $i$. We denote by $\bZ_v=\{0,1,\ldots, v-1\}$ the ring of integers modulo $v$.
There is a bijection from the set of all binary sequences of
length $v$ to the set of all subsets of $\bZ_v$ which assigns
to the sequence $A$ the subset $\{i\in\bZ_v:a_i=-1\}$.
If $X\subseteq\bZ_v$, then the corresponding binary sequence $[x_0,x_1,\ldots,x_{v-1}]$ has $x_i=-1$ if $i\in X$ and $x_i=+1$ otherwise. We associate to $X$ the cyclic matrix $C_X$ of order $v$ having this sequence as its first row.

Periodic Golay pairs are periodic analogs of the well known
Golay pairs. Let us give a precise definition. For any complex
sequence $A=[a_0,a_1,\ldots,a_{v-1}]$, its {\em periodic autocorrelation} is a complex valued function
$\paf_A:\bZ_v\to\bC$ defined by
\begin{equation}
\paf_A(s)=\sum_{j=0}^{v-1} a_{j+s}\bar{a}_j,
\end{equation}
where the indexes are computed modulo $v$ and $\bar{a}$ is the complex conjugate of $a$. A pair of binary sequences $(A,B)$ of length $v$ is a {\em periodic Golay pair} if
$\paf_A(j)+\paf_B(j)=0$ for $j\ne0$. For more information on these pairs see \cite{DK:arXiv:2013}. The length $v$ of a periodic Golay pair must be even except for the trivial case 
$v=1$,

Many periodic Golay pairs of even length $v$ can be constructed by using supplementary difference sets with suitable parameters $(v;r,s;\lambda)$. We recall that these parameters are nonnegative integers such that $\lambda(v-1)=r(r-1)+s(s-1)$. 
(See section \ref{SDS} below for the formal definition of SDSs over a finite cyclic group.) For convenience, we also introduce the parameter $n=r+s-\lambda$. Without any loss of generality we may assume that the parameter set is {\em normalized} which means that we have $v/2\ge r\ge s\ge0$.
The SDSs that we need are those for which $v=2n$. We refer to them as {\em periodic Golay SDS}.

The feasible parameter sets for the periodic Golay SDSs can be easily generated by using the following proposition. 

\begin{proposition} \label{Formule}
Let $P$ be the set of ordered pairs $(x,y)$ of integers $x,y$ 
such that $x\ge y\ge0$ and $x>0$. Let $Q$ be the set of normalized feasible parameter sets $(v;r,s;\lambda)$, with $v$ even, for periodic Golay SDSs. Thus, it is required that $v=2n$ where $n=r+s-\lambda$. Then the map $P\to Q$ given by the formula
$$
(x,y)\to ( 2(x^2+y^2);x^2+y^2-y,x^2+y^2-x; x^2+y^2-x-y )
$$
is a bijection. 
\end{proposition}
\begin{proof}
The inverse map $Q\to P$ is given by
$$
(v;r,s;\lambda)\to(\frac{v}{2}-s,\frac{v}{2}-r).
$$
\end{proof}

Note that $n=x^2+y^2$.

If $(A,B)$ is a periodic Golay pair of length $v$, then the corresponding pair of subsets $(X,Y)$ of $\bZ_v$ is an SDS. In the nontrivial cases $(v>1)$, the parameters $(v;r,s;\lambda)$ satisfy the equation $v=2n$. Recall that $n=r+s-\lambda$. The converse is also true, i.e., if $(X,Y)$ is an SDS with parameters
$(v;r,s;\lambda)$ then the corresponding binary sequences $(A,B)$ form a periodic Golay pair of length $v$. Moreover, if $a=v-2r$ and $b=v-2s$ then $a^2+b^2=2v$. In particular, $v$ must be even and a sum of two squares. The associated matrices $C_X$ and $C_Y$ satisfy the equation
\begin{equation} \label{MatEqPG}
C_X C_X^T + C_Y C_Y^T = 2vI_v.
\end{equation}

Our main result is the construction of several periodic Golay pairs of length 72. This is accomplished by constructing the SDSs with parameters $(72;36,30;30)$. The main tool that we use in the construction is the method of compression of SDSs developed in \cite{DK:JCD:2014}. This method uses a nontrivial factorization $v=md$ and so it can be applied only when $v$ is a composite integer. In this case we used the factorization with $m=3$ and $d=24$.

In section \ref{SDS} we recall the definition of SDSs over finite cyclic groups, and in section \ref{Compression} we establish a relationship between power spectral density functions of a complex sequence of length $v=md$ and its compressed sequence of length $d$. This relationship was used to speed up some of the computations.

In section \ref{sec:PG-results} we list 8 nonequivalent SDSs which give 8 periodic Golay pairs of length 72. This provides the first examples of periodic Golay pairs whose length is divisible by a prime congruent to 3 modulo 4.

\section{Supplementary difference sets} \label{SDS}

We recall the definition of SDSs. Let $k_1,\ldots,k_t$ be positive integers and $\lambda$ an integer such that
\begin{equation} \label{par-lambda}
\lambda (v - 1) = \sum_{i=1}^t k_i (k_i - 1).
\end{equation}

\begin{definition}
We say that the subsets $X_1,\ldots,X_t$ of $\bZ_v$ with
$|X_i|=k_i$ for $i\in\{1,\ldots,t\}$ are {\em supplementary difference sets (SDS)} with parameters $(v;k_1,\ldots,k_t;\lambda)$, if for every nonzero element $c\in\bZ_v$ there are exactly $\lambda$ ordered triples $(a,b,i)$ such that $\{a,b\}\subseteq X_i$ and $a-b=c \pmod{v}$.
\end{definition}

These SDS are defined over the cyclic group of order $v$, namely the additive group of the ring $\bZ_v$. More generally SDS can be defined over any finite abelian group, and there are also further generalizations where the group may be any finite group. However, in this paper we shall consider only the cyclic case.

In the context of an SDS, say $X_1,\ldots,X_t$, with parameters $(v;k_1,\ldots,k_t;\lambda)$, we refer to the subsets $X_i$ as the {\em base blocks} and we introduce an additional parameter, $n$, defined by:
\begin{equation} \label{par-n}
n = k_1 + \cdots + k_t - \lambda.
\end{equation}

If $x$ is an indeterminate, then the quotient ring
$\bC[x]/(x^v-1)$ is isomorphic to the ring of complex circulant matrices of order $v$. Under this isomorphism $x$ corresponds to the cyclic matrix with first row $[0,1,0,0,\ldots,0]$. By applying this isomorphism to the identity \cite[(13)]{DK:JCD:2014}, we obtain that the following matrix identity holds
\begin{equation} \label{MatNorm}
\sum_{i=1}^t C_i C_i^T = 4nI_v +(tv-4n)J_v,
\end{equation}
where $C_i=C_{X_i}$ is the cyclic matrix associated to $X_i$.

In this paper we are mainly interested in SDSs $(X,Y)$ with two base blocks $(t=2)$ and $v=2n$ Then the identity (\ref{MatNorm}) reduces to the identity (\ref{MatEqPG}).

\section{Compression of SDSs} \label{Compression}

Let $A$ be a complex sequence of length $v$. For the standard definitions of periodic autocorrelation functions $(\paf_A)$, discrete Fourier transform $(\dft_A)$, power spectral density $(\psd_A)$ of $A$, and the definition of complex complementary sequences, we refer the reader to \cite{DK:JCD:2014}. If we have a collection of complex complementary sequences of length $v=dm$, then we can compress them to obtain complementary sequences of length $d$. We refer to the ratio $v/d=m$ as the {\em compression factor}. Here is the precise definition.

\begin{definition}
Let $A = [a_0,a_1,\ldots,a_{v-1}]$ be a complex sequence of length $v = dm$ and set
\begin{equation} \label{koef-kompr}
a_j^{(d)}=a_j+a_{j+d}+\ldots+a_{j+(m-1)d}, \quad
j=0,\ldots,d-1.
\end{equation}
Then we say that the sequence
$A^{(d)} = [a_0^{(d)},a_1^{(d)},\ldots,a_{d-1}^{(d)}]$
is the {\em $m$-compression} of $A$.
\end{definition}

Let $X,Y$ be an SDS with parameters $(v;r,s;\lambda)$ with 
$v=2n$ (and $n=r+s-\lambda$). Assume that $v=md$ is a nontrivial factorization. Let $A,B$ be the binary sequences of length $v$ associated to $X$ and $Y$, respectively. Then the $m$-compressed sequences $A^{(d)},B^{(d)}$ form a complementary pair. In general they are not binary sequences, their terms belong to the set
$\{m,m-2,\ldots,-m+2,-m\}$.
The search for such pairs $X,Y$ is broken into two stages: first we construct the candidate complementary sequences $A^{(d)},B^{(d)}$ of length $d$, and second we lift each of them and search to find the required pairs $(X,Y)$. Each of the stages requires
a lot of computational resources. There are additional theoretical results that can be used to speed up these computations. Some of them are descirbed in \cite{DK:JCD:2014},
namely we use ``bracelets'' and ``charm bracelets'' to speed up
the first stage. We use \cite[Theorem 1]{DK:ADTHM:2014} to speed up the second stage.

\section{Multiplication of Golay and periodic Golay pairs} 
\label{GN}

If $Z\subseteq\bZ_v$ we set $Z'=\bZ_v\setminus Z$. To $Z$ we associate the binary sequence 
$[a_0,a_1,\ldots,a_{v-1}]$, where $a_i=-1$ if $i\in Z$ and $a_i=+1$ otherwise. This gives a one-to-one 
correspondence between subsets $Z\subseteq\bZ_v$ and the set of binary sequences of length $v$. 
If $(X,Y)$ is an SDS with parameters $(v;r,s;\lambda)$ such that $v=2n$, $(n=r+s-\lambda)$, 
then the associated binary sequences of $X$ and $Y$ form a periodic Golay pair. Conversely, each periodic 
Golay pair of length $v>1$ arises in this way from an SDS with $v=2n$.

If there exists a Golay pair resp. a periodic Golay pair of length $v$ then we say that 
$v$ is a {\em Golay number} resp. a {\em periodic Golay number}. We denote the set of Golay numbers by $\Gamma$ 
and the set of periodic Golay numbers by $\Pi$. By $\Gamma_0$ we denote the set of known Golay numbers, 
i.e., $\Gamma_0=\{2^a 10^b 26^c: a,b,c\in\bZ_+\}$, where $\bZ_+$ is the set of nonnegative integers.
It is not known whether $\Gamma_0=\Gamma$. Since every Golay pair is also a periodic Golay pair, we have 
$\Gamma\subseteq\Pi$. Moreover, this inclusion is strict. Indeed, the periodic Golay numbers $v=34,50,58,68,72,74,82$ (see 
\cite{{DK:arXiv:2013}}) are not in $\Gamma$ (see 
\cite{{Borwein:Ferguson:2003}}). 

If $X$ and $Y$ are sets of positive integers, we shall denote by $XY$ the set of all products $xy$ with 
$x\in X$ and $y\in Y$. Given a Golay pair of length $g$ and a periodic Golay pair of length $v$, then one can multiply them
to obtain a periodic Golay pair of length $gv$. In fact there are now two such multiplications which are essentially different.
Consequently, the set $\Pi\setminus\Gamma_0$ is infinite as it contains the set 
$\Gamma_0\cdot\{34,50,58,72,74,82,122,202,226\}$.

The first multiplication is described in the very recent paper \cite{GSDK:Bio:2014}. It is an easy consequence of \cite[Theorems 13,16]{KS}. We give below a simple description in terms of the SDS $(X,Y)$ associated to a periodic Golay pair. The parameters $(v;r,s;\lambda)$ and $n=r+s-\lambda$ of this SDS satisfy the equation $v=2n$. 

\begin{proposition} \label{Mnozenje-1}
Let $(U,V)$ be a Golay pair of length $g$ and $(X,Y)$ the SDS 
associated to a periodic Golay pair of length $v=2n$. 
Let $x,y$ be two indeterminates and define the sequence 
$A=[a_0,a_1,\ldots,a_{v-1}]$ by setting
$$
a_i=\cases{
x, &if $i\in X\cap Y$, \cr
-x, &if $i\in X'\cap Y'$, \cr
y, &if $i\in X\setminus Y$, \cr
-y, &if $i\in Y\setminus X$.}
$$
Next, let $B$ be the sequence obtained from $A$ by first reversing $A$ and then simultaneously replacing $x$ with $y$ and $y$ with $-x$. 
Finally, by replacing in both $A$ and $B$ the indeterminates $x$ and $y$ with $U$ and $V$, respectively, one obtains a periodic Golay pair of length $gv$.
\end{proposition}

We observed subsequently that Turyn's multiplication of Golay pairs provides also the multiplication of Golay and periodic Golay pairs. For convenience let us associate to each binary sequence $A=[a_0,a_1,\ldots,a_{v-1}]$ the polynomial 
$A(z):=a_0+a_1z+\cdots+a_{v-1}z^{v-1}$ in the indeterminate $z$. 
Then Turyn's multiplication $(A,B)\cdot(C,D)=(E,F)$ of 
Golay pairs $(A,B)$ of length $g$ and $(C,D)$ of length $v$ is given by the formulas (see \cite{RT})
\begin{eqnarray}
\label{jed-1}
E(z) &=& \frac{1}{2}[A(z)+B(z)]C(z^g) + 
\frac{1}{2}[A(z)-B(z)]D(z^{-g}) z^{gv-g}, \\
\label{jed-2}
F(z) &=& \frac{1}{2}[B(z)-A(z)]C(z^{-g}) z^{gv-g} + 
\frac{1}{2}[A(z)+B(z)]D(z^g). 
\end{eqnarray}
The product pair $(E,F)$ is a Golay pair of length $gv$. 

\begin{proposition} \label{Mnozenje-2}
If $(A,B)$ is a Golay pair of length $g$ and $(C,D)$ is a periodic Golay pair of length $v$ then the pair $(E,F)$ 
given by the formulas (\ref{jed-1}) and (\ref{jed-2}) is 
a periodic Golay pair of length $gv$.
\end{proposition}
\begin{proof}
The fact that $(A,B)$ is a Golay pair is equivalent to the 
identity
\begin{equation} \label{prva}
A(z)A(z^{-1}) + B(z)B(z^{-1}) = 2g.
\end{equation}
Similarly, the fact that $(C,D)$ is a periodic Golay pair is equivalent to the congruence
\begin{equation} \label{druga}
C(z)C(z^{-1}) + D(z)D(z^{-1}) \equiv 2v ~{\rm mod}~(z^v -1),
\end{equation}
where $(z^v -1)$ is the ideal of the Laurent polynomial ring 
$\bZ[z,z^{-1}]$ generated by $z^v -1$. 
A computation gives that

\begin{eqnarray*}
4E(z)E(z^{-1}) &=& 
(A(z)+B(z))(A(z^{-1})+B(z^{-1})) C(z^g)C(z^{-g}) +\\
&& (A(z)-B(z))(A(z^{-1})-B(z^{-1})) D(z^g)D(z^{-g}) +\\
&& (A(z)+B(z))(A(z^{-1})-B(z^{-1})) C(z^g)D(z^g) z^{g-gv} +\\
&& (A(z)-B(z))(A(z^{-1})+B(z^{-1})) C(z^{-g})D(z^{-g}) z^{gv-g},
\end{eqnarray*}

\begin{eqnarray*}
4F(z)F(z^{-1}) &=& 
(A(z)-B(z))(A(z^{-1})-B(z^{-1})) C(z^g)C(z^{-g}) +\\
&& (A(z)+B(z))(A(z^{-1})+B(z^{-1})) D(z^g)D(z^{-g}) +\\
&& (B(z)-A(z))(A(z^{-1})+B(z^{-1})) C(z^{-g})D(z^{-g}) 
z^{gv-g} +\\
&& (A(z)+B(z))(B(z^{-1})-A(z^{-1})) C(z^g)D(z^g) z^{g-gv}.
\end{eqnarray*}

By using (\ref{prva}) we obtain that
\begin{equation}
E(z)E(z^{-1}) + F(z)F(z^{-1})=
g(C(z^g)C(z^{-g}) + D(z^g)D(z^{-g})).
\end{equation}

It follows from (\ref{druga}) that
\begin{equation}
C(z^g)C(z^{-g}) + D(z^g)D(z^{-g})\equiv 2v ~
{\rm mod}~(z^{gv} -1),
\end{equation}
and so 
\begin{equation}
E(z)E(z^{-1}) + F(z)F(z^{-1}) \equiv 2gv ~{\rm mod}~(z^{gv} -1).
\end{equation}
This means that $(E,F)$ is a periodic Golay pair.
\end{proof}

As an example, let us take the Golay pair $(A=[-1,1],B=[1,1])$ 
of length $g=2$ and the periodic Golay pair $(C,D)$ of length 
$v=34$ with associated SDS $(X,Y)$ given by 
\begin{eqnarray*}
X &=& \{ 0,1,2,3,5,6,8,12,13,14,15,18,20,22,24,31 \}, \\
Y &=& \{ 0,1,4,5,7,8,9,14,15,18,23,26,28 \}.
\end{eqnarray*}
Its parameters are $(v=34;r=16,s=13;\lambda=12)$ and $n=17$. 
We compute the product $(E,F)=(A,B)\cdot(C,D)$ by using 
the multiplication from Propositions \ref{Mnozenje-1} and 
\ref{Mnozenje-2}. The associated SDSs $(P,Q)$ and $(R,S)$, 
respectively, are given by
\begin{eqnarray*}
P &=& \{ 0,2,8,9,10,14,15,16,18,19,21,23,28,30,33,35,36,39,43,46, \\
&& 47,51,52,53,55,56,57,59,61,65,67 \}, \\
Q &=& \{ 0,1,2,3,5,6,7,8,9,10,12,13,14,16,17,19,20,23,24,25,27, \\
&& 28,29,32,33,34,35,41,43,44,45,46,47,48,52,55,58,61,63\}, \\
R &=& \{ 4,5,7,13,18,21,22,23,25,26,27,30,33,35,36,38,39,40, \\
&& 41,42,43,45,49,50,51,54,55,56,59,60,61,62,63,64,65,66,67 \}, \\
S &=& \{ 1,3,5,7,10,11,13,14,17,20,25,27,29,30,31,36,37,38, \\
&& 41,45,48,49,50,52,56,58,63,64,66 \}. 
\end{eqnarray*}

After replacing $Q$ with its complement $Q'$ in $\bZ_v$, the parameters of these two SDS are $(68;31,29;26)$. However, one can verify that they are not equivalent as SDSs. Indeed, the  
canonical forms (see \cite{Djokovic:AnnComb:2011}) 
$(\hat{P},\hat{Q'})$ and $(\hat{R},\hat{S})$ of the SDSs 
$(P,Q')$ and $(R,S)$ are given by
\begin{eqnarray*}
\hat{P} &=& \{ 0,1,2,3,6,7,8,10,13,14,16,17,18,20,24,28,31,33,35,
 \\
&& 36,38,40,41,43,44,49,52,53,55,62,64 \}, \\
\hat{Q'} &=& \{ 0,1,2,3,4,5,11,12,13,16,18,19,20,23,24,25,26,28,30, \\
&& 36,39,41,42,45,50,51,55,60,64 \}, \\
\hat{R} &=& \{ 0,1,2,3,4,6,10,12,13,14,17,19,21,23,26,27,29,32,
 \\
&& 35,37,38,41,42,43,49,51,53,56,60,61,65 \}, \\
\hat{S} &=& \{ 0,1,3,4,5,6,8,9,10,11,13,15,16,20,23,24,26,27,28,
 \\
&& 36,38,41,44,45,50,51,52,57,58 \}. 
\end{eqnarray*}

It is rather surprising that the two multiplications described above produce nonequivalent periodic Golay pairs.

\section{Computational results for periodic Golay pairs}
\label{sec:PG-results}

No $v\in\Gamma$ is divisible by a prime congruent to 3 modulo 4 (see \cite{EKS:JCT-A:1990}). So far, none of the known members of $\Pi$ were divisible by a prime congruent to 3 modulo 4. Hence, the periodic Golay pairs constructed below are the first examples having the length divisible by a prime congruent to 3 modulo 4, namely the prime 3. Consequently, no periodic Golay pair of length 72 can be constructed by multiplying a nontrivial Golay pair and a periodic Golay pair.

We list eight pairwise nonequivalent SDSs with parameters $(72;36,30;30)$. As $n=36$ we have $v=2n$, and so these SDSs give periodic Golay pairs of length 72.
All solutions are in the canonical form defined in \cite{Djokovic:AnnComb:2011} and since they are different, this implies that they are pairwise non\-equivalent.

\begin{eqnarray*}
1) &&\{0,1,2,3,4,5,6,7,10,12,13,15,18,20,22,24,26,27,29,30,31,\\
&& 35,37,39,40,43,44,47,51,52,53,56,58,59,62,63\}, \\
&& \{0,1,2,3,5,6,8,11,12,13,14,15,18,21,23,25,29,32,33,39,41, \\
&& 42,43,47,48,55,56,62,67,69\}, \\
2) &&\{0,1,2,3,4,5,6,7,10,12,13,15,18,20,22,24,26,27,29,30,31,\\
&& 35,37,39,40,43,44,47,51,52,53,56,58,59,62,63\}, \\
&& \{0,2,3,5,7,8,9,11,14,15,17,18,19,23,24,30,31,32,33,37,38, \\
&& 41,42,44,48,49,51,59,61,69\}, \\
3) &&\{0,1,2,3,5,7,10,11,12,13,15,17,19,20,26,27,28,29,30,32,\\
&& 34,35,38,39,40,42,43,46,49,51,54,56,59,60,63,64\}, \\
&& \{0,1,2,3,4,6,7,8,9,14,15,16,20,22,24,26,27,31,33,36,37,40, \\
&& 42,43,46,49,54,57,58,68\}, \\
4) &&\{0,1,2,3,5,7,10,11,12,13,15,17,19,20,26,27,28,29,30,32, \\
&& 34,35,38,39,40,42,43,46,49,51,54,56,59,60,63,64\}, \\
&& \{0,1,3,4,6,7,8,9,10,14,15,18,19,20,22,25,26,31,32,36,38, \\
&& 40,42,45,49,51,52,57,58,60\}, \\
5) &&\{0,1,2,4,5,6,7,9,10,11,14,15,16,17,22,23,25,26,29,30,33,\\
&& 35,37,38,43,45,46,48,50,51,52,54,55,60,62,63\}, \\
&& \{0,2,3,5,7,8,9,11,14,17,18,19,21,23,24,27,30,31,32,37,38, \\
&& 41,42,44,48,49,57,59,61,63\}, \\
6) &&\{0,1,3,4,5,6,7,8,9,10,14,15,17,18,19,20,22,25,26,29,31, \\
&& 32,36,38,40,41,42,45,49,51,52,53,57,58,60,65\}, \\
&& \{0,1,2,5,7,10,11,12,13,17,19,20,26,28,29,30,32,34,35,38, \\
&& 40,42,43,46,49,54,56,59,60,64\}, \\
7) &&\{0,1,3,4,5,6,7,8,9,10,14,15,17,18,19,20,22,25,26,29,31,\\
&& 32,36,38,40,41,42,45,49,51,52,53,57,58,60,65\}, \\
&& \{0,1,3,4,5,6,9,10,13,16,18,19,21,23,24,27,30,34,35,40,46, \\
&& 47,48,49,53,55,57,63,65,67\}, \\
8) &&\{0,2,3,4,5,7,8,9,11,14,15,16,17,18,19,23,24,28,30,31,32,\\
&& 33,37,38,40,41,42,44,48,49,51,52,59,61,64,69\}, \\
&& \{0,1,2,4,5,6,7,10,12,13,18,20,22,24,26,29,30,31,35,37,40, \\
&& 43,44,47,52,53,56,58,59,62\}.
\end{eqnarray*}

Let $v\in\Pi$ and $v>1$. Then it is known that $v$ must be even and $v/2$ must be a sum of two squares. 
Moreover there is an SDS with parameters $(v;r,s;\lambda)$ such that $v=2n$. The Arasu-Xiang condition 
\cite[Corollary 3.6]{Arasu:Xiang:DCC:1992} for the existence of such SDS must be satisfied. This gives 
another restriction on $v$. 

The product $\Gamma_0 S$, where $S=\{1,34,50,58,72,74,82,122,202,226\}$, is the set of lengths of the currently known periodic Golay pairs. For reader's convenience we list the integers in the range $1<v\le300$ which 
satisfy all necessary conditions mentioned above and do not belong to $\Gamma_0 S$. There are just sixteen 
of them:
$$
90,106,130,146,170,178,180,194,212,218,234,250,274,290,292,298.
$$
These are the smallest lengths for which the existence question of periodic Golay pairs remains unsolved.

\section{Acknowledgements}
The authors wish to acknowledge generous support by NSERC.
This research was enabled in part by support provided by WestGrid 
(www.westgrid.ca) and Compute Canada Calcul Canada 
(www.computecanada.ca). We thank a referee for his suggestions,

\end{document}